\documentclass[10pt,a4paper,english]{scrartcl}
\usepackage{lmodern}

\usepackage[T1]{fontenc}
\usepackage[utf8]{inputenc}
\usepackage{mathtools}
\usepackage{amsmath}
\usepackage{amsthm}
\usepackage{amssymb}
\usepackage{graphicx}

\makeatletter

\theoremstyle{plain}
\newtheorem{thm}{\protect\theoremname}
\theoremstyle{definition}
\newtheorem{defn}[thm]{\protect\definitionname}
\theoremstyle{remark}
\newtheorem{rem}[thm]{\protect\remarkname}
\theoremstyle{plain}
\newtheorem{prop}[thm]{\protect\propositionname}
\theoremstyle{plain}
\newtheorem{cor}[thm]{\protect\corollaryname}
\theoremstyle{plain}
\newtheorem{lem}[thm]{\protect\lemmaname}
\newcommand{\lyxaddress}[1]{
	\par {\raggedright #1
	\vspace{1.4em}
	\noindent\par}
}

\makeatother

\usepackage{babel}
\providecommand{\corollaryname}{Corollary}
\providecommand{\definitionname}{Definition}
\providecommand{\lemmaname}{Lemma}
\providecommand{\propositionname}{Proposition}
\providecommand{\remarkname}{Remark}
\providecommand{\theoremname}{Theorem}

\begin{document}
\global\long\def\R{\mathbb{R}}%
\global\long\def\vu{\sqrt{1+|\nabla u|^{2}}}%
\global\long\def\divergence{\mathop{\mathrm{div}}}%
\global\long\def\graph{\mathop{\mathrm{graph}}}%
\global\long\def\L{\mathop{\mathcal{L}}}%
\global\long\def\ddt{\frac{\mathrm{d}}{\mathrm{d}t}}%

\title{Mean curvature flow of symmetric double graphs only develops singularities
on the hyperplane of symmetry}
\author{Wolfgang Maurer\thanks{funded by the Deutsche Forschungsgemeinschaft (DFG, German Research Foundation) Project number 336454636}}
\maketitle
\begin{abstract}
By a symmetric double graph we mean a hypersurface which is mirror-symmetric
and the two symmetric parts are graphs over the hyperplane of symmetry.
We prove that there is a weak solution of mean curvature flow that
preserves these properties and singularities only occur on the hyperplane
of symmetry. The result can be used to construct smooth solutions
to the free Neumann boundary problem on a supporting hyperplane with
singular boundary.

For the construction we introduce and investigate a notion named ``vanity''
and which is similar to convexity. Moreover, we rely on Sáez' and
Schnürer's ``mean curvature flow without singularities'' to approximate
weak solutions with smooth graphical solutions in one dimension higher.
\end{abstract}

\section{Introduction}

Symmetry is one of the most important concepts in mathematics. As
such it appears often and is worth studying. Our result demonstrates
that under certain circumstances one can delimit the region where
the occurence of singularities in mean curvature flow is possible
in a way naturally related to an existing symmetry. It is well-known
that singularities occur in the mean curvature flow of closed hypersurfaces
in Euclidean space. This can be easily seen by placing the closed
hypersurface inside a sphere, which shrinks to a point under the mean
curvature flow, and using the fact that by the comparison principle
the evolving hypersurface must stay inside the sphere. For graphical
hypersurfaces, however, the situation is different and no singularities
occur (cf.\ \cite{EH}). In this note, we investigate a situation
where these two regimes get in touch. Consider a mirror-symmetric
closed hypersurface, i.e., a hypersurface that is symmetric with respect
to some hyperplane. Furthermore, assume that the two symmetric parts
are graphical over that hyperplane. In the title we called such a
surface a symmetric double graph. Since the hypersurface is closed,
singularities inevitably arise for the mean curvature flow. But due
to the graphical properties, these occur only on the hyperplane of
symmetry. Or at least, we show that there exists a weak solution,
the singularity resolving solution of \cite{SS}, with this behavior.

One motivation for this problem arose from the need to approximate
a flow of non-compact graphs by flows of closed hypersurfaces. In
this context symmetric double graphs appear as a reasonable means
of approximation. In these situations it is highly desirable to exclude
singularities on the graphical parts, that is to say that singularities
only develop on the hyperplane of symmetry. Usually, this is done
by a convexity assumption because convex hypersurfaces don't develop
singularities until they vanish in a point. For the mean curvature
flow in particular, this is a well-known result of \cite{Huisken}.
Another interesting aspect is that the graphical parts may be viewed
as solutions of the free Neumann boundary value problem to mean curvature
flow where the hypersurfaces meet a supporting hyperplane perpendicularly.
Each symmetric part is a graphical solution to this problem that is
smooth up to the (free) boundary which may be singular.

To deal with the symmetric double graphs we introduce a notion which
is similar to convexity and which we call vanity. In a convex set
any two points can see each other; in a vain set each point can see
its mirror-image, hence the name. We will give a few results on vain
sets and functions which are reminiscent of corresponding results
for convex sets and functions.

We also heavily rely on the results of \cite{SS}. In that paper Sáez
and Schnürer established the existence for mean curvature flow of
complete graphs and interpreted the projections, the domains of definition
of the functions representing the graphs, as a weak solution for mean
curvature flow. They dubbed it a singularity resolving solution. The
complete graphs do not show any singularities, but the singularity
resolving solution can. So one has a weak solution at hand which flows
through singularities and which is backed by a smooth solution in
one dimension higher. We will exploit this and work with the smooth
graphical solutions to investigate the weak solution that appears
in the projections.

This note is organized as follows. In the next section we will introduce
the notion of vanity and prove a few related results. Next, in Section
\ref{sec:Graphical-mean-curvature} we investigate the graphical mean
curvature flow for vain functions. The results lay the foundation
for Section \ref{sec:Singularity-resolving-solutions}, where we show
that for a singularity resolving solution of symmetric double graphs,
the two graphical parts stay smooth. In particular, singularities
only appear on the hyperplane of symmetry.

\section{Vanity}
\begin{defn}
\global\long\def\ox{\overline{x}}%
 For $(x^{1},\ldots,x^{n})\in\R^{n}$ we denote by $\ox$ the reflection
of $x$ in the first direction. So $\ox$ is given by 
\[
\ox\coloneqq(-x^{1},x^{2},\ldots,x^{n})\;.
\]

Moreover, we will write 
\[
x_{\lambda}\coloneqq(\lambda x^{1},x^{2},\ldots,x^{n})=\frac{x+\ox}{2}+\lambda\,\frac{x-\ox}{2}\text{ for }\lambda\in[-1,1]\,.
\]
\end{defn}

\begin{rem}
We have $x_{1}=x$ and $x_{-1}=\ox$.
\end{rem}

\begin{defn}
\, 
\begin{enumerate}
\item A subset $\Omega\subset\R^{n}$ is called \emph{vain} if all of its
points can see their mirror image, i.e., 
\[
\forall x\in\Omega\,\forall\lambda\in[-1,1]\colon x_{\lambda}\in\Omega\qquad\text{holds.}
\]
\item If $\Omega\subset\R^{n}$ is vain, then a function $u\colon\Omega\to\overline{\R}=[-\infty,+\infty]$
is called \emph{vain} if 
\[
\forall x\in\Omega\,\forall\lambda\in[-1,1]\colon u(x_{\lambda})\le u(x)\qquad\text{holds.}
\]
\end{enumerate}
\end{defn}

\begin{rem}
\label{rem easy consequences} \, 
\begin{enumerate}
\item Every vain set is mirror-symmetric. 
\item Every vain function $u\colon\Omega\to\overline{\R}$ is mirror-symmetric,
i.e., $u(x)=u(\ox)$ holds for all $x\in\Omega$. \label{it ec sym} 
\item If $\Omega$ is vain, then the fibers of the projection $\Omega\ni x\mapsto(x^{2},\ldots,x^{n})$
are convex, i.e., they are lines.\label{it ec fibers} 
\item A function $u\colon\Omega\to\overline{\R}$ is vain if and only if
the set 
\[
\big\{(x^{1},\ldots,x^{n+1})\in\Omega\times\overline{\R}\colon u(x^{1},\ldots,x^{n})\le x^{n+1}\big\}
\]
is vain. 
\end{enumerate}
\end{rem}

\begin{prop}
\label{prop vain level-sets} If $u\colon\Omega\to\overline{\R}$
is vain, then the sets $u^{-1}([-\infty,a])$ and $u^{-1}([-\infty,a))$
are vain for any $a\in\overline{\R}$.
\end{prop}

\begin{proof}
Obvious from the definitions. 
\end{proof}
\begin{prop}
\label{prop vain monotonical} If $u$ is a vain function and $m$
is a monotonically increasing function, then $m\circ u$ is vain. 

Furthermore, if $m\colon\overline{\R}^{k}\to\overline{\R}$ is a monotonically
increasing function with respect to the partial ordering given by
$\left(z_{1},\ldots,z_{k}\right)\le\left(z_{1}',\ldots z_{k}'\right)\,\iff\,\forall i\in\{1,\ldots,k\}:z_{i}\le z_{i}'$
and if $u_{1},\ldots,u_{k}$ are vain functions on a vain set $\Omega$,
then $\Omega\ni x\mapsto m\left(u_{1}(x),\ldots,u_{k}(x)\right)$
is vain, too.
\end{prop}

\begin{proof}
The inequalities $u_{i}(x_{\lambda})\le u_{i}(x)$ immediately imply
$m\left(u_{1}(x_{\lambda}),\ldots,u_{k}(x_{\lambda})\right)\le m\left(u_{1}(x),\ldots,u_{k}(x)\right)$.
\end{proof}
\begin{cor}
\label{cor vain sum} For two vain functions $u,v\colon\Omega\to\R$,
their sum $u+v$ is vain again.
\end{cor}

\begin{prop}
Let $\Omega$ be a vain set. A function $u\colon\Omega\to\overline{\R}$
is vain if and only if $u$ is mirror-symmetric and for any $(x^{2},\ldots,x^{n})\in\R^{n-1}$,
$u(\cdot,x^{2},\ldots,x^{n})$ is monotonically increasing on $\{x^{1}\colon(x^{1},x^{2},\ldots,x^{n-1})\in\Omega,\,x^{1}\ge0\}$. 
\end{prop}

\begin{proof}
Firstly, let $u$ be vain. Then $u$ is mirror-symmetrical (Remark
\ref{rem easy consequences} (\ref{it ec sym})). Let $(x^{2},\ldots,x^{n})\in\R^{n-1}$.
Let $0\le a\le b$ be such that $x^{(b)}\coloneqq(b,x^{2},\ldots,x^{n})\in\Omega$,
and consequently $x^{(a)}\coloneqq(a,x^{2},\ldots,x^{n})\in\Omega$.
With $\lambda\coloneqq\frac{a}{b}\in[0,1]$ we can write $x^{(a)}=x_{\lambda}^{(b)}$.
Hence, by the vanity of $u$, $u(x^{(a)})\le u(x^{(b)})$ holds, which
proves the monotonicity of $u(\cdot,x^{2},\ldots,x^{n})$ on the set
$\{x^{1}\colon(x^{1},x^{2},\ldots,x^{n})\in\Omega,\,x^{1}\ge0\}$.

Now we assume the symmetry and the monotonicity property and prove
that $u$ is vain. Let $x\in\Omega$ and $\lambda\in[-1,1]$. We shall
prove $u(x_{\lambda})\le u(x)$. Symmetry is the reason why we only
need to consider $x^{1}\ge0$ and $\lambda\ge0$. Then $u(x_{\lambda})\le u(x)$
follows from the monotonicity property.
\end{proof}
\begin{cor}
\label{cor partial_1 u ge 0} Let $u\in C^{1}(\Omega)$ be a mirror-symmetrical
function on a vain set $\Omega$. Then $u$ is vain if and only if
$\partial_{1}u(x)\ge0$ for $x^{1}>0$ holds.
\end{cor}

The following explains the relation of vanity to symmetric double
graphs.
\begin{cor}
\label{prop vain implies graph} Let $\Omega$ be an open, vain set.
Then it is of the form 
\begin{equation}
\Omega=\big\{(x^{1},\hat{x})\in\R^{n}\colon\hat{x}\in U,\,|x^{1}|<h(\hat{x})\big\}\label{eq prop vain implies graph}
\end{equation}
for a function $h$ defined on an open set $U\subset\R^{n-1}$. 
\end{cor}

\begin{proof}
Let $U$ be the projection of $\Omega$ to $\R^{n-1}$, where we use
the projection $p$ that drops the first coordinate. Because $\Omega$
is open and vain, the fibers of that projection $p$ are lines of
the form 
\[
p^{-1}(\hat{x})=\{\hat{x}\}\times(-h(\hat{x}),h(\hat{x}))\;.
\]
Defining $h$ on $U$ in this way shows that $\Omega$ has the asserted
form.
\end{proof}
\begin{rem}
For any continuous function $h$ on an open subset $U\subset\R^{n-1}$,
we can define an open, vain set via (\ref{eq prop vain implies graph}).
\end{rem}

\begin{lem}
\label{lem vain dist} For a vain set $\Omega\subset\R^{n}$ the negative
distance function $-d\coloneqq-\mathrm{dist}_{\partial\Omega}$, defined
on $\Omega$, is vain.
\end{lem}

\begin{proof}
Let $x\in\Omega$ and $\lambda\in[-1,1]$ be given. Because $\Omega$
is mirror-symmetrical, we have $d(x)=d(\ox)$. For $v\in\R^{n}$ with
$|v|<d(x)=d(\ox)$, there hold $x+v\in\Omega$ and $\ox+v\in\Omega$.
Noting Remark \ref{rem easy consequences} (\ref{it ec fibers}),
we deduce $x_{\lambda}+v\in\Omega$. Because $v$ was an arbitrary
vector subject to the condition $|v|<d(x)$, it follows $d(x_{\lambda})\ge d(x)$.
\end{proof}

\paragraph{Mollification of Vain Functions.}

We check whether the standard mollification by convolution with a
mollifying kernel preserves vanity. For special kernels we can give
an affirmative answer. However, this is not expected for general kernels.
For instance, have a look at the vain function $\sqrt{|x|}$ and the
``kernel'' $\frac{1}{2}(\delta_{\varepsilon}+\delta_{-\varepsilon})$.
Then the convolution is not vain.

Let $\eta\in C_{c}^{\infty}(\R^{n})$ be a Friedrichs mollifier such
that $-\eta$ is vain. By Corollary \ref{cor partial_1 u ge 0} this
is equivalent to the conditions $\eta(\overline{x})=\eta(x)$ for
all $x\in\R^{n}$ and $\partial_{1}\eta(x)\ge0$ for $x^{1}<0$. For
$\varepsilon>0$ we define $\eta_{\varepsilon}(x)\coloneqq\varepsilon^{-n}\,\eta(x/\varepsilon)$.
Clearly, this function has these same properties as $\eta$ does.

Let $u\colon\R^{n}\to\R$ be a vain function. The mollification of
$u$ is defined by $u_{\varepsilon}(x_{0})\coloneqq\int u(x)\,\eta_{\varepsilon}(x_{0}-x)\,\mathrm{d}x$.
We check $u_{\varepsilon}(\overline{x}_{0})=u_{\varepsilon}(x_{0})$
and $\partial_{1}u_{\varepsilon}(x_{0})\ge0$ for $x_{0}^{1}>0$:
(We use transformation variables $y=\overline{x}$ and $z=\overline{x}+2\,x_{0}^{1}\,e_{1}$.)
\begin{equation}
\begin{split}u_{\varepsilon}(\overline{x}_{0}) & =\int u(x)\,\eta_{\varepsilon}(\overline{x}_{0}-x)\,\mathrm{d}x=\int u(\overline{y})\,\eta_{\varepsilon}(\overline{x}_{0}-\overline{y})\,\mathrm{d}y\\
 & =\int u(y)\,\eta_{\varepsilon}(x_{0}-y)\,\mathrm{d}y=u_{\varepsilon}(x_{0})\;,
\end{split}
\label{eq vain molli sym}
\end{equation}
\begin{equation}
\begin{split}\partial_{1}u_{\varepsilon}(x_{0}) & =\int u(x)\,\partial_{1}\eta_{\varepsilon}(x_{0}-x)\,\mathrm{d}x\\
 & =\int_{\{x^{1}<x_{0}^{1}\}}u(x)\,\partial_{1}\eta_{\varepsilon}(x_{0}-x)\,\mathrm{d}x+\int_{\{x^{1}>x_{0}^{1}\}}u(x)\,\partial_{1}\eta_{\varepsilon}(x_{0}-x)\,\mathrm{d}x\\
 & =\int_{\{z^{1}>x_{0}^{1}\}}u(\overline{z}+2\,x_{0}^{1}\,e_{1})\,\partial_{1}\eta_{\varepsilon}(\underbrace{x_{0}-\overline{z}-2\,x_{0}^{1}\,e_{1}}_{\overline{x}_{0}-\overline{z}})\,\mathrm{d}z\\
 & \quad+\int_{\{x^{1}>x_{0}^{1}\}}u(x)\,\partial_{1}\eta_{\varepsilon}(x_{0}-x)\,\mathrm{d}x\\
 & =\int_{\{x^{1}>x_{0}^{1}\}}\big(u(x)-u(\overline{x}+2\,x_{0}^{1}\,e_{1})\big)\,\partial_{1}\eta_{\varepsilon}(x_{0}-x)\,\mathrm{d}x\;.
\end{split}
\label{eq vain molli}
\end{equation}
We have used $\partial_{1}\eta_{\varepsilon}(\overline{x_{0}-z})=-\partial_{1}\eta_{\varepsilon}(x_{0}-z)$
and have renamed the integration variable from $z$ to $x$ in the
last step. For $0<x_{0}^{1}<x^{1}$ we have $-x^{1}<\overline{x}^{1}+2\,x_{0}^{1}<\overline{x}^{1}+2\,x^{1}=x^{1}$
and $\partial_{1}\eta_{\varepsilon}(x_{0}-x)\ge0$. We deduce from
the vanity of $u$ that $u(x)\ge u(\overline{x}+2\,x_{0}^{1}\,e_{1})$
and therefore obtain from (\ref{eq vain molli}) 
\[
\partial_{1}u_{\varepsilon}(x_{0})\ge0\qquad\text{for \ensuremath{x_{0}^{1}>0}\,.}
\]
Together with (\ref{eq vain molli sym}), the symmetry of $u_{\varepsilon}$,
the vanity of the mollification $u_{\varepsilon}$ now follows from
Corollary \ref{cor partial_1 u ge 0}.

\section{Graphical mean curvature flow\label{sec:Graphical-mean-curvature}}

For a function $(x,t)\mapsto u(x,t)$ of space and time, its graphs
$\left(u(\cdot,t)\right)_{t}$ flow by mean curvature if and only
if $u$ solves the graphical mean curvature flow
\[
\partial_{t}u=\left(\delta^{ij}-\frac{u^{i}\,u^{j}}{1+|\mathrm{D}u|^{2}}\right)u_{ij}\,.
\]
In this section we examine whether vanity is preserved under the graphical
mean curvature flow. Proposition \ref{prop vain MCF} proves that
for a bounded situation with constant boundary values. In Proposition
\ref{prop vain MCFwS} we demonstrate that for the graphical mean
curvature flow of complete graphs there exists a solution which preserves
vanity.
\begin{prop}
\label{prop vain MCF} Let $\Omega\subset\R^{n}$ be a vain, open,
bounded, and smooth set. Let $u\in C^{2;1}(\Omega\times[0,T))$ be
a solution of graphical mean curvature flow such that $u\le a$ ($a\in\R$)
and $u(x,t)=a$ for $x\in\partial\Omega,\,t\in[0,T)$ hold.

If $u(\cdot,0)$ is vain, then $u(\cdot,t)$ is vain for all $t\in[0,T)$.
\end{prop}

\begin{proof}
Because of the uniqueness of the solution, $u(\cdot,t)$ is mirror-symmetrical
for all $t\in[0,T)$.

Let $\nu(\cdot,t)$ be the downwards pointing normal to $M_{t}\coloneqq\graph u(\cdot,t)\subset\R^{n+1}$.
If we track points on $M_{t}$ in time along the normal direction,
then on $M_{t}$ holds
\[
\partial_{t}\nu-\Delta\nu=|A|^{2}\,\nu\;,
\]
where $|A|^{2}$ is the squared norm of the second fundamental form
and $\Delta$ denotes the Laplace-Beltrami operator of $M_{t}$. Accordingly,
for $w\coloneqq\nu_{1}=\left\langle \nu,e_{1}\right\rangle \equiv\frac{\partial_{1}u}{\sqrt{1+|\mathrm{D}u|^{2}}}$
holds 
\begin{equation}
\partial_{t}w-\Delta w=|A|^{2}\,w\;.\label{eq vain eq w}
\end{equation}

Let $M_{t}^{+}\coloneqq\{X\in M_{t}\colon X^{1}>0\}$. By the mirror-symmetry
of $u(\cdot,t)$, we have $w(\cdot,t)=0$ on the part of $\partial M_{t}^{+}$
which lies on $\{X^{1}=0\}$. The vanity of $\Omega$ and $u(\cdot,t)\le a$
as well as $u(x,t)=a$ for $x\in\partial\Omega$ imply that $w(\cdot,t)\ge0$
holds on the remaining part of $\partial M_{t}^{+}$ (cf.\ Proposition
\ref{prop vain implies graph}). Furthermore, the vanity of $u(\cdot,0)$
implies that $w(\cdot,0)\ge0$ on $M_{0}^{+}$ (cf.\ Corollary \ref{cor partial_1 u ge 0}).

The parabolic maximum principle yields $w(p,t)\ge0$ for all $p\in M_{t}^{+}$
and for all $t\in[0,T)$. This in turn implies the vanity of $u(\cdot,t)$
for all $t\in[0,T)$ (Corollary \ref{cor partial_1 u ge 0}).
\end{proof}
\begin{prop}
\label{prop vain MCFwS} Let $\Omega_{0}\subset\R^{n}$ be an open,
vain set and let $u_{0}\colon\Omega_{0}\to\R$ be a vain, locally
Lipschitz function such that for any $a\in\R$ the set $\{x:u_{0}(x)\le a\}$
is compact.

Then there exists a relatively open set $\Omega=\bigcup_{t\ge0}\Omega_{t}\times\{t\}\subset\R^{n}\times[0,\infty)$
compatible with $\Omega_{0}$ from above and such that $\Omega_{t}$
is vain for every $t\ge0$. And there exists a continuous function
$u\colon\Omega\to\R$ which is smooth on $\bigcup_{t>0}\Omega_{t}\times\{t\}$
and which is a solution of the graphical mean curvature flow starting
from the initial condition $u_{0}$ such that $u(\cdot,t)\colon\Omega_{t}\to\R$
is vain for every $t\ge0$. Moreover, for any $a\in\R$ the set $\left\{ (x,t):u(x,t)\le a\right\} $
is compact.
\end{prop}

\begin{rem}
The assumption that $\{x:u_{0}(x)\le a\}$ be compact implies the
completeness of the graph of $u_{0}$ as well as a boundedness of
$u_{0}$ from below. The compactness of $\left\{ (x,t):u(x,t)\le a\right\} $
combines the completeness of the graphical hypersurfaces with a maximality
condition: One cannot stop the solution at some arbitrary time without
destroying that property.
\end{rem}

\begin{proof}
One uses the construction for the mean curvature flow without singularities
from \cite{SS} and checks whether the vanity is preserved in the
steps taken there.

Firstly, we extend $u_{0}$ to a function $\overline{u}_{0}\colon\R^{n}\to\overline{\R}$
by setting $\overline{u}_{0}(x)=+\infty$ for $x\notin\Omega_{0}$.
Then $\overline{u}_{0}$ is vain. Next, we cut off at some height
$a\in\R$ by concatenation with the monotonically increasing function
$\min\{\cdot,a\}$. By Proposition \ref{prop vain monotonical} $\min\{\overline{u}_{0},a\}$
is a vain function. Because $\{x:u_{0}(x)\le a\}$ is compact, the
resulting function $\min\{\overline{u}_{0},a\}$ is constantly equal
to $a$ outside a ball. A subsequent mollification as described in
the last section does not destroy the vanity nor this constancy property.

Having prepared the initial data in this way we can consider with
these the Dirichlet problem on a large ball for the graphical mean
curvature flow. This gives us an approximating solution. Proposition
\ref{prop vain MCF} ensures that the approximating solution is still
vain.

Of course we let $a\to\infty$, let the mollification parameter tend
to zero, and take larger and larger balls. The resulting sequence
of approximating solutions converges pointwise to a spacetime function
$\overline{u}\colon\R^{n}\times[0,\infty)\to\overline{\R}$. In fact,
the convergence is smooth on sets $\left\{ (x,t):\overline{u}(x,t)\le a\right\} $,
$a\in\R$ and $u\coloneqq\overline{u}|_{\Omega}$ with $\Omega\coloneqq\{(x,t):\left|\overline{u}(x,t)\right|<\infty\}$
is as asserted except for the vanity assertions (cf. \cite{SS}).

We have $\Omega=\overline{u}^{-1}((-\infty,\infty))=\overline{u}^{-1}([-\infty,\infty))$.
Therefore, the vanity of $\Omega$ and thus the vanity of the time
slices $\Omega_{t}$ follows from Proposition \ref{prop vain level-sets}.
Because the property of being vain is preserved under pointwise limits,
$\overline{u}(\cdot,t)$ is vain for any $t\ge0$. Hence, the same
holds for $u(\cdot,t)$.
\end{proof}

\section{Singularity resolving solutions\label{sec:Singularity-resolving-solutions}}
\begin{cor}
\label{cor vain weak solution} Let $\Omega_{0}\subset\R^{n}$ be
open and vain. Then there is a weak solution $(\Omega_{t})_{t\in[0,\infty)}$
of the mean curvature flow in the sense of a singularity resolving
solution that starts with $\Omega_{0}$ and such that $\Omega_{t}$
is vain for all $t\ge0$.
\end{cor}

\begin{rem}
Singularity resolving solutions are introduced in \cite{SS}. A singularity
resolving solution is the projection, or likewise the domains of definition,
of a mean curvature flow of complete graphical hypersurfaces. It has
been proven in \cite{SS} that in the case of non-fattening of the
level-set flow starting from the boundary $\partial\Omega_{0}$, $\Omega_{t}$
coincides $\mathcal{H}^{n}$-almost everywhere with the corresponding
set of the level-set flow. In \cite{Shadows}, the author has shown
that $\left(\Omega_{t}\right)_{t\ge0}$ is always a weak solution
in the sense that any smooth mean curvature flow starting inside of
$\Omega_{t_{0}}$ will stay inside of $\Omega_{t}$ for $t\ge t_{0}$
and analogous for ``outside''.
\end{rem}

\begin{proof}
Without loss of generality we assume $\Omega_{0}\neq\emptyset$. Let
$d\coloneqq\mathrm{dist}_{\partial\Omega_{0}}$ be the positive distance
function to the boundary on $\Omega_{0}$. We set $u_{0}(x)=\frac{1}{d(x)}+|x|^{2}$
for $x\in\Omega_{0}$. By Lemma \ref{lem vain dist} and Propositions
\ref{prop vain monotonical} and \ref{cor vain sum}, $u_{0}$ is
a vain function and $u_{0}$ satisfies the hypothesis of Proposition
\ref{prop vain MCFwS}. Let $u\colon\Omega\to\R$ be a solution from
this Proposition (uniqueness of the solution is not proven). Then
the $\Omega_{t}$ are vain and they form a singularity resolving solution
by definition.
\end{proof}
\begin{thm}
\label{thm vain N} Let $\Omega_{0}\subset\R^{n}$ be an open and
vain set. Suppose $\Omega_{0}$ is bounded in the $x^{1}$-direction,
$N_{0}\coloneqq\partial\Omega_{0}\cap\{x^{1}>0\}$ is of class $C^{2}$,
and suppose that the curvature of $N_{0}$ is bounded on sets $\{x:x^{1}>\varepsilon\}$
and that $\nu_{1}\ge c>0$ is positively bounded from below on these
sets. (The bounds may depend on $\varepsilon$. $\nu$ is the normal
to $N_{0}$.)

Then, for the weak solution $(\Omega_{t})_{t\in[0,\infty)}$ from
Corollary \ref{cor vain weak solution}, the $N_{t}\coloneqq\partial\Omega_{t}\cap\{x:x^{1}>0\}$
are smooth submanifolds for $t>0$ and they solve the mean curvature
flow in the classical sense.
\end{thm}

\begin{proof}
Let $u\colon\Omega\to\R$ be the mean curvature flow without singularities
from the proof of Corollary \ref{cor vain weak solution}. Let $M_{t}\coloneqq\graph u(\cdot,t)$.
We will prove uniform estimates for $M_{t}$ in the region $X^{1}\ge\varepsilon$.
From these we infer that $\left(M_{t}\cap\{x:x^{1}>0\}\right)-j\,e_{n+1}\xrightarrow{j\to\infty}N_{t}\times\R$
converges locally smoothly with locally uniform estimates. In particular,
$N_{t}$ is smooth.

For the estimates one would like to use the cut-off function $(X^{1}-\varepsilon)_{+}$.
However, this function doesn't cut off compact subsets from the mean
curvature flow $M_{t}$. For this reason, one considers cut-off functions
similar to those in \cite{EH} and whose supports are given by shrinking
balls. One chooses larger and larger balls such that in the limit
the half-space $X^{1}>\varepsilon$ is obtained. More precisely, we
do the following construction. For $R_{0}>0$ and $\varepsilon>0$,
we set $X_{0}\coloneqq(R_{0}+\varepsilon,0,\ldots,0)\in\R^{n+1}$.
We define the corresponding cut-off function $\varphi\colon\R^{n+1}\times[0,\infty)\to\R$
by 
\[
\varphi(X,t)\coloneqq\frac{1}{2\,R_{0}}\left(R_{0}^{2}-2\,n\,t-|X-X_{0}|^{2}\right)_{+}\,.
\]
The support of $\varphi$ is given by a shrinking ball around $X_{0}$
of initial radius $R_{0}$. If $(X,t)$ is fixed and $R_{0}\to\infty$,
in the limit one obtains the cut-off function 
\[
\begin{split}\lim_{R_{0}\to\infty}\varphi(X,t) & =\lim_{R_{0}\to\infty}\frac{1}{2\,R_{0}}\left(-2\,n\,t-\sum_{\alpha\ge2}(X^{\alpha})^{2}+R_{0}^{2}-(X^{1}-(R_{0}+\varepsilon))^{2}\right)_{+}\\
 & =\lim_{R_{0}\to\infty}\left(\frac{1}{2\,R_{0}}\big(R_{0}-(X^{1}-(R_{0}+\varepsilon))\big)\big(R_{0}+(X^{1}-(R_{0}+\varepsilon))\big)\right)_{+}\\
 & =\lim_{R_{0}\to\infty}\left(\frac{1}{2\,R_{0}}\,(2\,R_{0}-X^{1}+\varepsilon)\,(X^{1}-\varepsilon)\right)_{+}\\
 & =(X^{1}-\varepsilon)_{+}\;.
\end{split}
\]

Considering the operator $\ddt-\Delta$ on the surface $M_{t}$ and
where $\varphi>0$, the function $\varphi$ satisfies, with a local
parametrization $p\mapsto X(p,t)$ of $M_{t}$ such that $\partial_{t}X$
points in normal direction, 
\global\long\def\BHO{\left(\ddt-\Delta\right)}%
 
\begin{equation}
\begin{split}\BHO\varphi\big(X(p,t),t\big) & =\frac{1}{2\,R_{0}}\left(-2\,n-\BHO|X(p,t)|^{2}\right)\\
 & =\frac{1}{2\,R_{0}}\left(-2\,n-2\left\langle \BHO X,X\right\rangle +2\,|\nabla X|^{2}\right)\\
 & =\frac{1}{2\,R_{0}}\,(-2\,n+2\,n)=0\;.
\end{split}
\label{eq vain eq phi}
\end{equation}
Here we have used $\BHO X=0$ and $|\nabla X|^{2}=g^{ij}\,\delta_{\alpha\beta}\,\nabla_{i}X^{\alpha}\,\nabla_{j}X^{\beta}=g^{ij}\,g_{ij}=n$.

In a region where $\varphi>\delta\,t$ holds ($\delta>0$), we have
for $\psi\coloneqq\varphi-\delta\,t$ 
\begin{equation}
\BHO\log\psi=\frac{\BHO\psi}{\psi}+\left|\frac{\nabla\psi}{\psi}\right|^{2}=-\frac{\delta}{\psi}+\left|\frac{\nabla\psi}{\psi}\right|^{2}\;.\label{eq vain eq psi}
\end{equation}

At first, we estimate $w=\nu_{1}$, where $\nu$ is the downwards
pointing normal to $M_{t}$. By the vanity of $u$, we have $w\ge0$
in the region $X^{1}\ge0$. The strong maximum principle implies that
$w>0$ in the region $X^{1}>0$ (note (\ref{eq vain eq w})).

In an interior maximum point of $w^{-1}\,\psi$, and consequently
of $-\log w+\log\psi$, there hold $\frac{\nabla\psi}{\psi}=\frac{\nabla w}{w}$
and 
\[
\begin{split}0 & \le\BHO(-\log w+\log\psi)\\
 & =-|A|^{2}-\left|\frac{\nabla w}{w}\right|^{2}-\frac{\delta}{\psi}+\left|\frac{\nabla\psi}{\psi}\right|^{2}\qquad\text{(note \eqref{eq vain eq w} and \eqref{eq vain eq psi})}\\
 & =-|A|^{2}-\frac{\delta}{\psi}<0\;.
\end{split}
\]
Contradiction. So there cannot be an interior maximum point, and because
$\psi$ vanishes on the lateral boundary of the region $\{(X,t):\varphi(X,t)>\delta\,t\}$,
$w^{-1}\,\psi$ is bounded by the supremum of its initial values.
With $\delta\to0$ it follows that $w^{-1}\,\varphi$ is bounded by
its initial values, too. Finally, we let $R_{0}\to\infty$ and obtain
the estimate 
\begin{equation}
w^{-1}\,(X^{1}-\varepsilon)_{+}\le\sup_{t=0}w^{-1}\,(X^{1}-\varepsilon)_{+}\;.\label{eq vain w}
\end{equation}

\global\long\def\uw{\underline{w}}%
 For the curvature estimate we consider the test function $f\coloneqq\frac{|A|\,\varphi}{w-\uw}$
with $\uw=\frac{1}{2}\inf w$, where the infimum is taken over the
set $\mathrm{supp}\,\varphi$. Using (\ref{eq vain w}) with $\tilde{\varepsilon}=\frac{\varepsilon}{2}$,
we see that $\uw\ge c>0$.

We shall make use of (\ref{eq vain eq w}), (\ref{eq vain eq phi}),
\begin{align*}
\left(\frac{\mathrm{d}}{\mathrm{d}t}-\Delta\right)|A|^{2} & =-2\,|\nabla A|^{2}+2\,|A|^{4},\\
\text{and}\quad\left|\nabla|A|^{2}\right| & =\left|2\left\langle A,\nabla A\right\rangle \right|\le2\,|A|\,|\nabla A|\,.
\end{align*}
In an interior maximum point of $f$, there holds
\[
\begin{split}0 & \le\BHO\log f^{2}=\BHO\left(\log|A|^{2}-2\log(w-\uw)+2\log\varphi\right)\\
 & =\underbrace{-2\,\frac{|\nabla A|^{2}}{|A|^{2}}}_{\le-\frac{1}{2}\left|\frac{\nabla|A|^{2}}{|A|^{2}}\right|^{2}}+2\,\frac{|A|^{4}}{|A|^{2}}+\left|\frac{\nabla|A|^{2}}{|A|^{2}}\right|^{2}-2\,\frac{w}{w-\uw}\,|A|^{2}-2\left|\frac{\nabla w}{w-\uw}\right|^{2}+2\left|\frac{\nabla\varphi}{\varphi}\right|^{2}\\
 & \le\frac{1}{2}\left|\frac{\nabla|A|^{2}}{|A|^{2}}\right|^{2}-\frac{2\,\uw}{w-\uw}\,|A|^{2}-2\left|\frac{\nabla w}{w-\uw}\right|^{2}+2\left|\frac{\nabla\varphi}{\varphi}\right|^{2}\,.
\end{split}
\]
 For arbitrary $\delta>0$ we deduce from the condition $\nabla\log f^{2}=0$
at the maximal point 
\[
\left|\frac{\nabla|A|^{2}}{|A|^{2}}\right|^{2}=\left|2\,\frac{\nabla w}{w-\uw}-2\,\frac{\nabla\varphi}{\varphi}\right|^{2}\le4\,(1+\delta)\left|\frac{\nabla w}{w-\uw}\right|^{2}+4\,(1+\delta^{-1})\left|\frac{\nabla\varphi}{\varphi}\right|^{2}\,.
\]
Therefore, 
\begin{equation}
0\le-\frac{2\,\uw}{w-\uw}\,|A|^{2}+2\,\delta\left|\frac{\nabla w}{w-\uw}\right|^{2}+2\,(2+\delta^{-1})\left|\frac{\nabla\varphi}{\varphi}\right|^{2}\,.\label{eq vain after extremal cond}
\end{equation}
There hold 
\begin{equation}
|\nabla\varphi|^{2}=\left|\left\langle \frac{X-X_{0}}{R_{0}},\nabla X\right\rangle \right|^{2}\le1\cdot|\nabla X|^{2}=g^{ij}\,\delta_{\alpha\beta}\,\nabla_{i}X^{\alpha}\,\nabla_{j}X^{\beta}=g^{ij}\,g_{ij}=n
\end{equation}
and 
\begin{equation}
|\nabla w|^{2}=|\nabla\nu^{1}|^{2}=g^{ij}\,(h_{i}^{k}\,\nabla_{k}X^{1})\,(h_{j}^{l}\,\nabla_{l}X^{1})\le|A|^{2}\,|\nabla X^{1}|^{2}\le|A|^{2}\,.
\end{equation}
Substituting these two inequalities into (\ref{eq vain after extremal cond})
yields 
\begin{equation}
0\le\left(\frac{-2\,\uw}{w-\uw}+\frac{2\,\delta}{(w-\uw)^{2}}\right)|A|^{2}+2\,(2+\delta^{-1})\,\frac{n}{\varphi^{2}}\;.
\end{equation}
With the choice $\delta=\frac{1}{2}\uw^{2}$, we obtain 
\begin{equation}
\left(\frac{-2\,\uw}{w-\uw}+\frac{2\,\delta}{(w-\uw)^{2}}\right)=\frac{-2\,\uw\,(w-\uw)+\uw^{2}}{(w-\uw)^{2}}\le\frac{-2\,\uw^{2}+\uw^{2}}{(w-\uw)^{2}}=-\frac{\uw^{2}}{(w-\uw)^{2}}\;.
\end{equation}
We conclude 
\begin{equation}
\frac{|A|^{2}\,\varphi^{2}}{(w-\uw)^{2}}\le2\,(2+2\,\uw^{-2})\,n\,\uw^{-2}\equiv C(\uw,n)\;.
\end{equation}
So, in an interior maximum point, $f$ is bounded by a controlled
constant. In particular, it follows that $|A|\,\varphi\le C(\uw,n)\,\left(1+\sup_{t=0}\left(|A|\,\varphi\right)\right).$
With $R_{0}\to\infty$ we obtain the estimate 
\begin{equation}
|A|\,(X^{1}-\varepsilon)_{+}\le C\,\left(\sup_{t=0}\left(|A|\,(X^{1}-\varepsilon)_{+}\right),\,\sup_{X^{1}>\varepsilon}w^{-1},\,n\right).
\end{equation}
Together with the estimate (\ref{eq vain w}) on $w$, we obtain the
curvature estimate 
\begin{equation}
\sup_{X^{1}>2\varepsilon}|A|\le C\left(\varepsilon,\,\sup_{t=0,\,X^{1}>\varepsilon}|A|,\,\sup_{t=0,\,X^{1}>\frac{1}{2}\varepsilon}w^{-1},\,\sup_{t=0}X^{1},\,n\right).\label{eq vain |A|}
\end{equation}

The higher order estimates, i.e., estimates on $|\nabla^{k}A|$, are
omitted because these can be obtained by following \cite{EH} from
hereon.

We have proven uniform estimates for $M_{t}=\graph u(\cdot,t)$ for
all $t$ in regions $\{X:X^{1}\ge\varepsilon\}$. By Proposition \ref{prop vain implies graph},
$M_{t}\cap\{X^{1}>0\}$ is graphical over the hyperplane orthogonal
to $e_{1}$. Let $h$ be the representing function. Then $h$ is bounded
and $h$ is monotonically increasing in the $e_{n+1}$ direction.
The inequality (\ref{eq vain w}) yields a gradient estimate on $h$
where $h>\varepsilon$. The estimates on the second fundamental form
(\ref{eq vain |A|}) and on its higher derivatives provide us with
estimates on the higher derivatives of $h$. We conclude that there
is a smooth limit of $h(\cdot,x^{n+1})$ for $x^{n+1}\to\infty$.
This limit is a graphical representation of $N_{t}$. Hence, $N_{t}$
is smooth.
\end{proof}
\begin{rem}
\label{rem vain freebdry} As a byproduct, Theorem \ref{thm vain N}
(see also Proposition \ref{prop vain implies graph}) provides a solution
of a free boundary value problem where the boundary moves on a hyperplane
and the hypersurface meets that hyperplane perpendicularly. Namely,
the family $(N_{t})_{t\in[0,\infty)}$ is a family of smooth graphical
hypersurfaces over a hyperplane that solves the mean curvature flow
($N_{t}$ may be empty). The boundaries $\partial N_{t}$, which reside
on the hyperplane, may be singular, however. But at spacetime-points
(on the hyperplane) where $\partial\Omega_{t}$ is smooth $\partial N_{t}$
is smooth too and by symmetry of $\partial\Omega_{t}$ the normal
to $\partial\Omega_{t}$ lies in the hyperplane such that $N_{t}=\partial\Omega_{t}\cap\{x:x^{1}>0\}$
meets the hyperplane perpendicularly. In this way $N_{t}$ can be
viewed as a smooth graphical solution to the free Neumann boundary
value problem with singularities at the boundary.

As was pointed out to the author by O.~Schnürer, it is not clear
that $N_{t}$ always meets the hyperplane perpendicularly. To explain
this, one needs to think about singularities of $\partial\Omega_{t}$
where it is possible to continuously extend the normal coming from
one of the two sides but where the limits from the two sides disagree.
In this case it may be possible that the normal for $N_{t}$ is definable
on the hyperplane but that it points out of the hyperplane.

\begin{figure}
\includegraphics[width=0.5\textwidth]{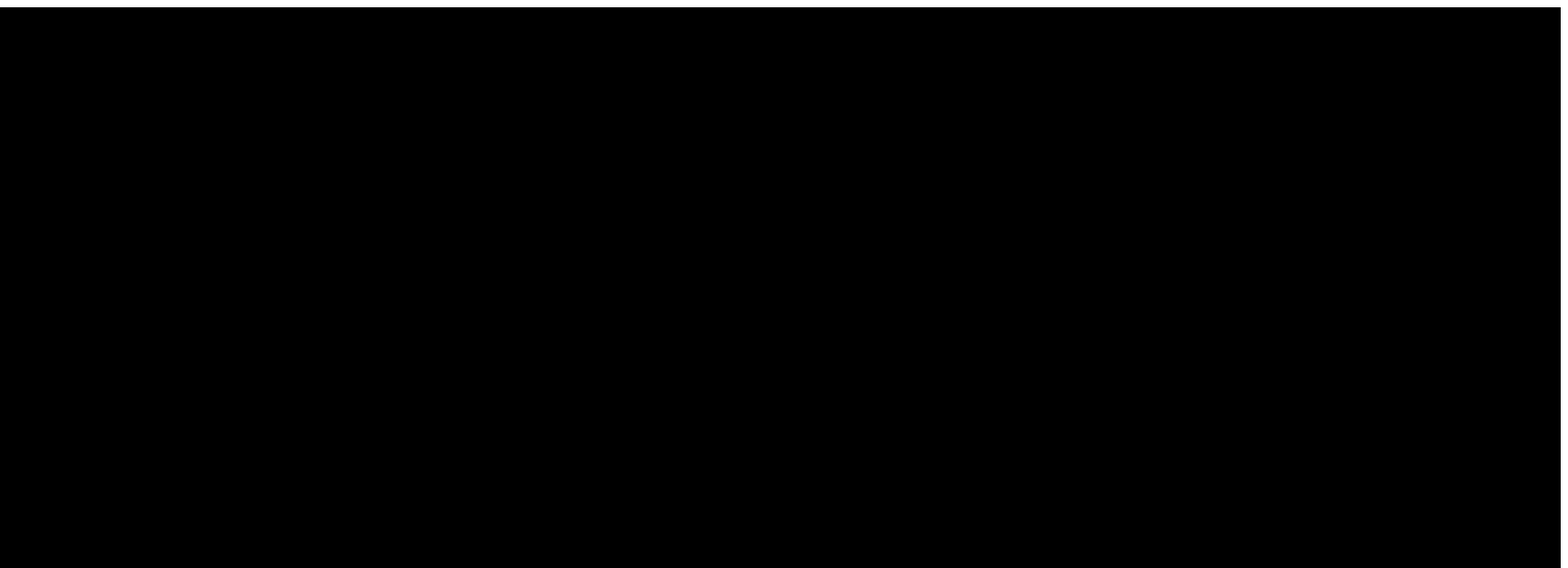}\hfill{}\includegraphics[width=0.5\textwidth]{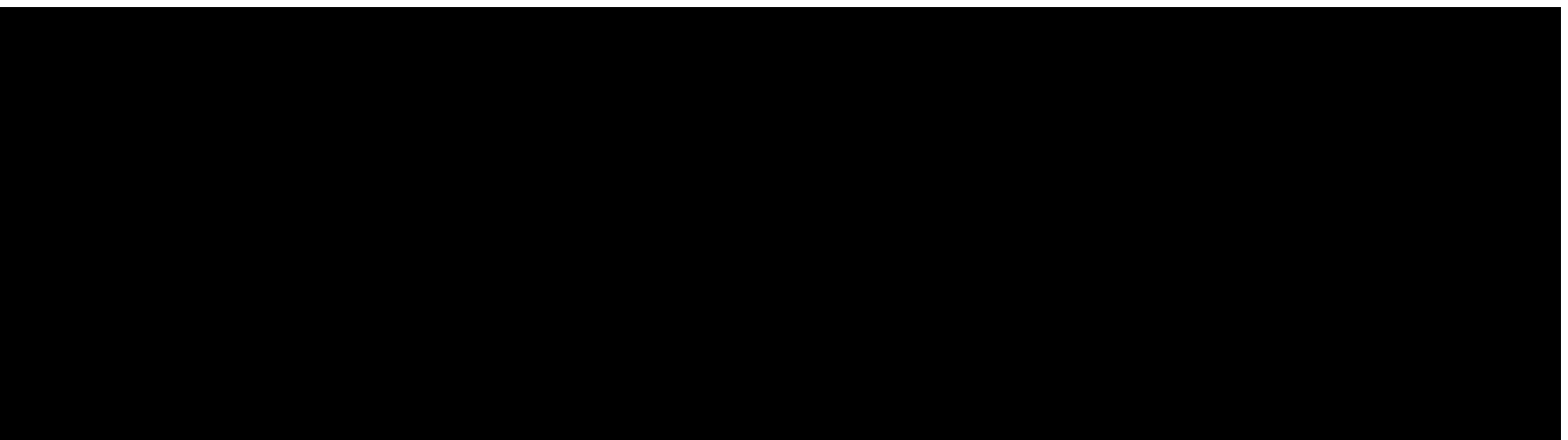}

\caption{Free boundary value problem}
The graphical surface moves by mean curvature flow subject to the
condition of meeting the plane perpendicularly. The first picture
shows the situation before, the second after the formation of a singularity.
Despite the singularity at the boundary the surface stays the graph
of a function which is smooth in the interior for all time.
\end{figure}
\end{rem}

\lyxaddress{\begin{center}
Wolfgang A.\ Maurer, Fachbereich Mathematik und Statistik, Universität
Konstanz, 78457 Konstanz, Germany\\
e-mail: wolfgang.maurer@uni-konstanz.de
\par\end{center}}
\end{document}